\documentclass[11pt]{article}

\usepackage{latexsym}
\usepackage{mathrsfs}
\usepackage{amsmath,amsthm}
\usepackage{graphicx,epstopdf,epsfig,multirow,epic,bm}
\usepackage{amssymb}
\usepackage{color}
\usepackage{colortbl}

\theoremstyle{plain}
\newtheorem{thm}{Theorem}

\newtheorem{lem}[thm]{Lemma}

\theoremstyle{definition}
\newtheorem{defn}{Definition}
\theoremstyle{definition}
\newtheorem{prob}{Problem}
\theoremstyle{plain}
\newtheorem{conj}{Conjecture}
\theoremstyle{plain}

\normalsize \rm
\DeclareGraphicsExtensions{.eps,.eps.gz}
\theoremstyle{plain}

\newcommand{\js}{\hfill $\square$}
\newcommand{\ud}{\mathrm{deg}}

\usepackage[displaymath,mathlines]{lineno}

\begin{document}

\thispagestyle{empty}

\begin{center}
{\huge\bf Probing Graph Proper Total Colorings With Additional
Constrained Conditions}\\[12pt]
{\large Bing Yao$^{a,}$\footnote{Corresponding author, Email: yybb918@163.com}\quad Ming Yao$^{b}$\quad Xiang-en Chen$^{a}$}
\end{center}
\begin{flushleft}
{\footnotesize a. College of Mathematics and Statistics, Northwest
Normal University, Lanzhou, 730070, CHINA\\
b. Department of Information Process and Control Engineering
Lanzhou Petrochemical College of Vocational Technology
Lanzhou, 730060, CHINA}
\end{flushleft}

\begin{abstract}
Graph colorings are becoming an
increasingly useful family of mathematical models for a broad range
of applications, such as \emph{time tabling and scheduling},
\emph{frequency assignment}, \emph{register allocation},
\emph{computer security} and so on. Graph proper total colorings with additional constrained conditions
have been investigated intensively in the last decade year.  In this article some new graph proper total
colorings with additional constrained conditions are defined, and
approximations to the chromatic numbers of these colorings are
researched, as well as some graphs having these colorings have been
verified.\\[6pt]
\end{abstract}

\section{Introduction and concepts}

A graph coloring/labelling is an assignments to vertices, edges or both by some certain requirements. The main reason is that graph colorings can divide a complex network into some smaller subnetworks such that each subnetwork has itself character differing from that of the rest subnetworks. Graph colorings/labellings have been applied in many areas of science and mathematics, such as in  X-ray crystallographic analysis, to design communication network, in determining optimal circuit layouts and radio astronomy. In \cite{Jaina-Krishna2002} the authors pointed out that graph theory provides important tools to capture various aspects of the network structure, and the analysis of such dynamical systems is facilitated by the development of some new tools in graph theory. It is interesting that the frequency assignment problem of  communication networks  is very similar with graph distinguishing colorings below. In the article \cite{Burris-Schelp}, Burris and Schelp introduce that a proper edge-coloring of  a simple graph $G$ is called a \emph{vertex distinguishing edge-coloring} (vdec) if for any two distinct vertices $u$ and $v$ of $G$, the set of the colors assigned to the edges incident to $u$ differs from the set of the colors assigned to the edges incident to $v$. The minimum number of colors required for all vertex distinguishing colorings of $G$ is denoted by $\chi \,'_s(G)$. Let $n_d=n_d(G)$ denote the number of all vertices of degree $d$ in $G$. It is clear that ${{\chi \,'_s(G)}\choose d}\geq n_d$ for all $d$ with respect to $\delta(G)\leq d\leq \Delta(G)$. Burris and Schelp \cite{Burris-Schelp} presented the following conjecture:

\begin{conj} \label{conj:c4-Burris-Schelp-conjecture}
Let $G$ be a simple graph with no isolated edges and at most one isolated vertex, and let $k$ be the smallest integer such that $(^k_d)\geq n_d$ for all $d$ with respect to $\delta(G)\leq d\leq
\Delta(G)$. Then $k\leq \chi \,'_s(G)\leq k+1$.
\end{conj}

A weak version of the vdec was introduced in \cite{Zhang-Liu-Wang-2001}, called the \emph{adjacent vertex distinguishing
edge coloring} (avdec). Zhang et al. \cite{Zhang-Liu-Wang-2001} asked for every edge $xy$ of $G$, the set of the colors assigned to the edges incident to $x$ differs from the set of the colors assigned to the edges incident to $y$ in an avdec, and use the notation $\chi \,'_{as}(G)$ to denote the least number of $k$ colors required fir which $G$ admits a $k$-avdec. They proposed: Every simple graph $G$ having no isolated edges  and at most one isolated vertex holds $\chi \,'_{as}(G)\leq \Delta(G)+2$. Surprisingly, it is very difficult to
settle down this conjecture, even for simple graphs (cf. \cite{Balister-Bollobas-Schelp}). In 2005, Zhang et al.\cite{Zhang-Chen-Li-Yao-Lu-Wang2005} investigated the adjacent vertex distinguishing total coloring (avdtc) of graphs, and proposed a conjecture: $\chi \,''_{as}(G)\leq \Delta(G)+2$, where $\chi \,''_{as}(G)$ is the smallest number of $k$ colors for which $G$ admits a $k$-avdtc. But, settling down these two conjectures is not a light work. Graph  distinguishing colorings are investigated intensively within two decades years (cf.  \cite{Dong-Wang2012}, \cite{Dong-Wang2014}, \cite{Hatami-H2010}, \cite{Yu-Qu-Wang-Wan2016}).

We use standard notation and terminology of graph theory. The shorthand symbol $[a,b]$ denotes an integer set
$\{a,a+1,a+2,\dots, b\}$ with integers $b>a\geq 1$. The set of vertices adjacent to a vertex $u$ is denoted by $N(u)$, and the set of edges incident to the vertex $u$ is denoted by $N_e(u)$. We call a graph $G$ to be simple if the degree $\ud_G(u)=|N(u)|$ for every vertex $u\in V(G)$. Graphs mentioned here are simple, undirected and finite.  Let $f$ be a proper total $k$-coloring of a simple graph $G$. The colors of neighbors of the vertex $u$ form the following color sets $C(f,u)=\{f(e): e\in N_e(u)\}=\{f(ux):x\in N(u)\}$, $C\langle f,u\rangle=\{f(x): x\in N(u)\}\cup \{f(u)\}$, $C[f,u]=C(f,u)\cup
\{f(u)\}$, and $N_2[f,u]=C(f,u)\cup C\langle f,u\rangle$. Notice that $\ud_G(u)+1\leq |N_2[f,u]|$, where $\ud_G(u)=|N(u)|$ is the degree of the vertex $u$. These color sets gives rise to distinguishing total colorings of various types. So we have a set $A_{cc}(G)$ containing the following \emph{additional constrained conditions}:

(C\textbf{1}) $C(f,u)\neq C(f,v)$ for distinct $u,v\in V(G)$;

(C\textbf{2}) $C(f,x)\neq C(f,y)$ for every edge $xy\in E(G)$;

(C\textbf{3}) $C\langle f,u\rangle\neq C\langle f,v\rangle$ for
distinct $u,v\in V(G)$;

(C\textbf{4}) $C\langle f,x\rangle\neq C\langle f,y\rangle$ for
every edge $xy\in E(G)$;

(C\textbf{5}) $C[f,u]\neq C[f,v]$ for distinct $u,v\in V(G)$;

(C\textbf{6}) $C[f,x]\neq C[f,y]$ for every edge $xy\in E(G)$;

(C\textbf{7}) $N_2[f,u]\neq N_2[f,v]$ for distinct $u,v\in V(G)$;
and

(C\textbf{8}) $N_2[f,x]\neq N_2[f,y]$ for every edge $xy\in E(G)$.

\vskip 0.2cm

We restate some known colorings again and define new colorings in
Definition \ref{def:N2-mu-e-mixed-distinguishing-coloring}.

\begin{defn} \label{def:22}
Let $f$ be a proper total $k$-coloring of a simple graph $G$ having
$n\geq 3$ vertices and no isolated edges as well as at most one
isolated vertex. We have eight distinguishing total colorings with
additional constrained conditions as follows:

\emph{Type 1}: This total $k$-coloring $f$ is called an
\emph{e-partially vertex distinguishing proper total $k$-coloring}
(e-partially $k$-vdtc) if it holds (C\textbf{1}), and the smallest
number of $k$ colors required for which $G$ admits an e-partially $k$-vdtc is
denoted as $\chi\,''_{(s)}(G)$; this total coloring $f$ is called an
\emph{e-partially adjacent vertex distinguishing proper total
$k$-coloring} (e-partially $k$-avdtc) if it holds (C\textbf{2}), and
the smallest number of $k$ colors required for which $G$ admits an
e-partially $k$-avdtc is denoted as $\chi\,''_{(as)}(G)$.

\emph{Type 2}: This total $k$-coloring $f$ is called a \emph{v-partially
vertex distinguishing proper total $k$-coloring} (v-partially
$k$-vdtc) if it holds (C\textbf{3}), and the smallest number of $k$
colors required for which $G$ admits a v-partially $k$-vdtc is denoted as
$\chi\,''_{\langle s\rangle}(G)$; this total coloring $f$ is called
a \emph{v-partially adjacent vertex distinguishing proper total
$k$-coloring} (v-partially $k$-avdtc) if it holds (C\textbf{4}), and
the smallest number of $k$ colors required for which $G$ admits a v-partially
$k$-avdtc is denoted as $\chi\,''_{\langle as\rangle}(G)$.

\emph{Type 3}: This total $k$-coloring $f$ is called a \emph{vertex
distinguishing proper total $k$-coloring} ($k$-vdtc) if it holds
(C\textbf{5}), and the smallest number of $k$ colors required for which $G$
admits a $k$-vdtc is denoted as $\chi\,''_{s}(G)$; this total
coloring $f$ is called an \emph{adjacent vertex distinguishing
proper total $k$-coloring} ($k$-avdtc) if it holds (C\textbf{6}),
and the smallest number of $k$ colors required for which $G$ admits a
$k$-avdtc is denoted as $\chi\,''_{as}(G)$.

\emph{Type 4}: This total $k$-coloring $f$ is called a $\mu
(k)$-\emph{coloring} if it holds (C\textbf{7}), and the notation
$\chi\,''_{2s}(G)$ stands for the least number of $k$ colors
required for which $G$ admits a $\mu (k)$-coloring; this total
coloring $f$ is called a $\mu _e(k)$-\emph{coloring} if it holds
(C\textbf{8}), and the symbol $\chi\,''_{2as}(G)$ denotes the least
number of $k$ colors required for which $G$ admits a
$\mu_e(k)$-coloring.\js
\end{defn}

Clearly, the degree $\ud_G(u)$ of a vertex $u$ of a simple graph $G$ holds: $\big |C(f,u)\big |=\ud_G(u)$, $2\leq \big |C\langle f,u\rangle \big |\leq \ud_G(u)+1$, $\big |C[f,u]\big |=\ud_G(u)+1$ and $\ud_G(u)+1\leq \big |N_2[f,u]\big |\leq 2\ud_G(u)+1$. Therefore, v-partially $k$-vdtcs, v-partially $k$-avdtcs, $\mu (k)$-colorings and $\mu _e(k)$-colorings may be complicated than e-partially $k$-vdtcs, e-partially $k$-avdtcs, $k$-vdtcs and $k$-avdtcs. In Figure \ref{fig:N2-total-01}(a) and (b), we can see $\chi\,''_{(s)}(G)<\chi\,''_{(s)}(H)$ although $H$ is a proper
subgraph of $G$.

\begin{defn} \label{def:33}
Let $f$ be a proper total coloring of a simple graph $G$ having $n\geq 3$ vertices and no isolated edges as well as at most one isolated vertex. We call $f$ an \emph{$(8)$-distinguishing total coloring} if it holds each one of $A_{cc}(G)$. The minimum number of $k$ colors required for which $G$ admits an $(8)$-distinguishing total
$k$-coloring is denoted as $\chi\,''_{(8)}(G)$. We call $f$ a \emph{$(6)$-distinguishing total coloring} if it holds each one of $A_{cc}(G)\setminus \{$ (C\textbf{3}), (C\textbf{4})$\}$. The minimum number of $k$ colors required for which $G$ admits a $(6)$-distinguishing total $k$-coloring is denoted as $\chi\,''_{(6)}(G)$.\js
\end{defn}

The $(8)$-distinguishing total coloring has been discussed in \cite{Yang-Yao-Ren2015} and \cite{Yang-Ren-Yao2016}. Furthermore, we can define: A \emph{$(4)$-avdtc} (resp. a $(3)$-avdtc) is a proper total
coloring holding each one of $A_{cc}(G)\setminus \{ (C\textbf{1}), (C\textbf{3}), (C\textbf{5})  (C\textbf{7})\}$  (resp. $A_{cc}(G)\setminus \{$ (C\textbf{1}), (C\textbf{3}), (C\textbf{4}), (C\textbf{5}) (C\textbf{7})$\}$). And $\chi\,''_{(4)as}(G)$ (resp. $\chi\,''_{(3)as}(G)$) is the minimum number of $k$ colors required
for which $G$ admits a $k$-$(4)$-avdtc (resp. $k$-$(3)$-avdtc). A \emph{mixed $(m)$-total coloring} with $m\in [2, 7]$ is a proper total coloring holding $m$ additional constrained conditions of $A_{cc}(G)$, herein.

\begin{figure}[h]
\centering
\includegraphics[height=3.7cm]{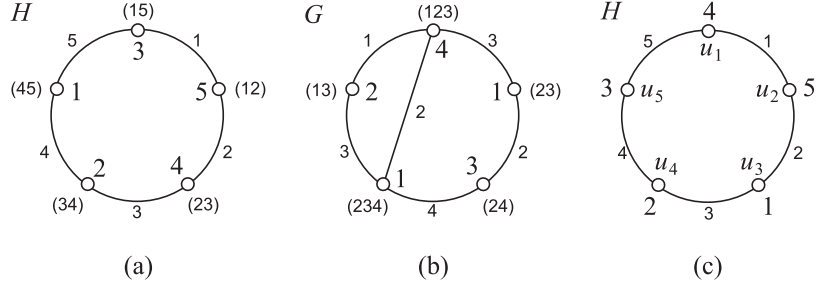}
\caption{\label{fig:N2-total-01}  {\small (a) $H$ admits an
e-partially $5$-vdtc (also, an e-partially $5$-avdtc), and
$\chi\,''_{(s)}(H)=5=\chi\,''_{(as)}(H)$; (b) $G$ admits an
e-partially $4$-vdtc (also, an e-partially $4$-avdtc), and
$\chi\,''_{(s)}(G)=4=\chi\,''_{(as)}(G)$; (c) $H$ admits an
$(8)$-distinguishing total $5$-coloring $f$, and
$\chi\,''_{(8)}(H)=5$.}}
\end{figure}

Clearly, every graph admits $(6)$-distinguishing total $k$-colorings
and $(3)$-avdtcs.  In Figure \ref{fig:N2-total-01}(c), all color
sets of the graph $H$ are listed in the following Table-1.
$$
\textbf{Table-1:\quad }
\begin{array}{c|cccc}
m&C(f,u_i)&C\langle f,u_i\rangle &C[f,u_i]&N_2[f,u_i]\\
\hline
u_1&\{1,5\}&\{3,4,5\}&\{1,4,5\}&\{1,3,4,5\}\\
u_2&\{1,2\}&\{1,4,5\}&\{1,2,5\}&\{1,2,4,5\}\\
u_3&\{2,3\}&\{1,2,5\}&\{1,2,3\}&\{1,2,3,5\}\\
u_4&\{3,4\}&\{1,2,3\}&\{2,3,4\}&\{1,2,3,4\}\\
u_5&\{4,5\}&\{2,3,4\}&\{3,4,5\}&\{2,3,4,5\}
\end{array}
$$

\begin{figure}[h]
\centering
\includegraphics[height=4.2cm]{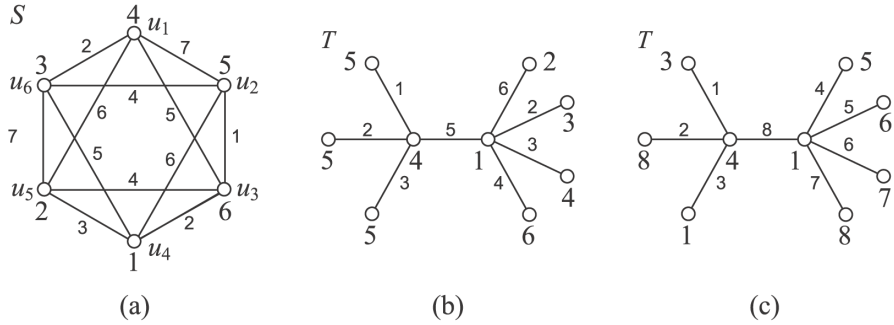}
\caption{\label{fig:N2-total-02}  {\small  (a) $S$ admits an
$(8)$-distinguishing total coloring, and $\chi\,''_{(8)}(S)=7$; (b)
$T$ admits a $6$-$(4)$-avdtc, and $\chi\,''_{(4)as}(T)=6$; (c) $T$
admits an $(8)$-distinguishing total coloring, and
$\chi\,''_{(8)}(T)=n_1(T)+1=8.$}}
\end{figure}

\begin{figure}[h]
\centering
\includegraphics[height=4.5cm]{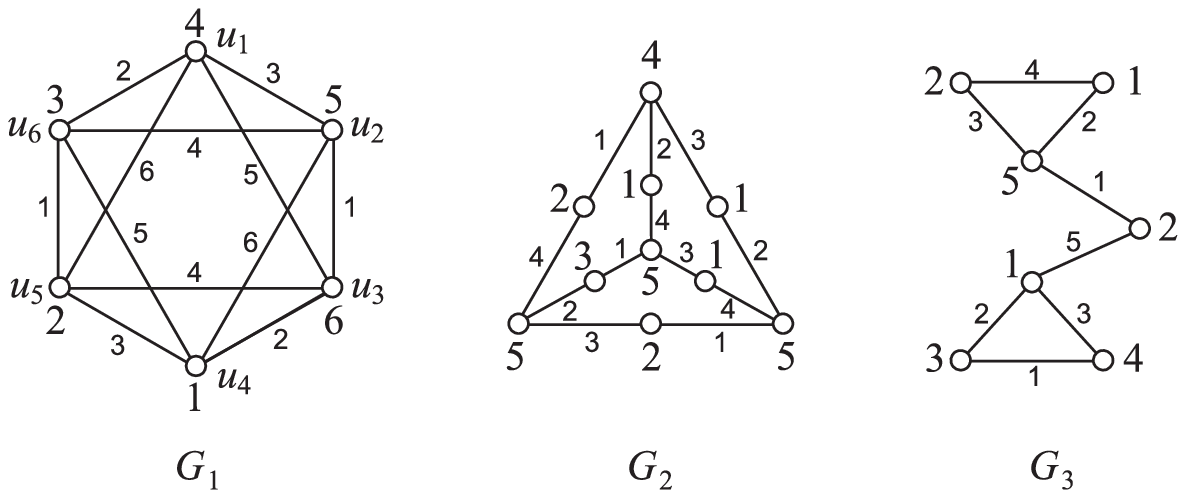}
\caption{\label{fig:n-suns}  {\small $G_1$ admits a mixed
$(7)$-total coloring; $G_2$ admits an $(8)$-distinguishing total
coloring, and $\chi\,''_{(8)}(G_2)=5$; $G_3$ admits a $(3)$-avdtc.}}
\end{figure}

Notice that $\chi\,''_{2s}(K_n)=\chi\,''_{2as}(K_n)$. In our memory,
the exact value $\chi\,''_{2as}(K_n)$ is not determined for all
integers $n\geq 3$ up to now. In the article
\cite{Zhang-Cheng-Yao-Li-Chen-Xu2008}, the authors introduce an
adjacent vertex strong-distinguishing proper total colorings that
is a $\mu (k)$-coloring defined in Definition
\ref{def:N2-mu-e-mixed-distinguishing-coloring}. They obtain the
exact values of $\chi\,''_{2as}$ for cycles $C_n$, paths $P_n$ and trees, and
present

\begin{conj} \label{conj:c4-Burris-Schelp-conjecture}
Let $G$ be a simple
graph with $n\geq 3$ vertices and no isolated edges. Then
$\chi\,''_{2as}(G)\leq n+\lceil \log_2n\rceil +1$.
\end{conj}

The procedure of ``joining a vertex of a graph $G$ and a vertex $v$
out of $G$ by an edge'' is abbreviated as ``adding a leaf $v$ to
$G$'' here.  The notations $\chi(G)$, $\chi \,'(G)$, $\chi \,''(G)$,  denote the \emph{proper
chromatic number}, \emph{proper chromatic index}  and \emph{proper total chromatic number} of $G$
(cf. \cite{Bondy-Murty-new}), respectively.

\section{Graphs having total colorings with additional constrained conditions}

For the purpose of convenience, let $\mathcal {F}_{3s}(n)$ be the set of simple graphs with $n\geq 3$ vertices and no isolated edges as well as at most one isolated vertex. The following Lemma \ref{thm:observation000}
follows from Definition \ref{def:N2-mu-e-mixed-distinguishing-coloring}.

\begin{lem} \label{thm:observation000}
Let $G\in \mathcal {F}_{3s}(n)$. Then

$(i)$ $\chi\,''_{(s)}(G)\geq \chi\,'_{s}(G)$, $\chi\,''_{(s)}(G)\geq
\chi\,''_{(as)}(G)$ and $\chi\,''_{(as)}(G)\geq \chi\,'_{as}(G)$.

$(ii)$ $\chi\,''_{2s}(G)\geq \chi\,''_{2as}(G)\geq \chi\,''(G)$.

$(iii)$ If $n_d\geq 2$ with $\delta(G)\leq d\leq \Delta(G)$, then $
\chi\,''_{2s}(G)\geq d+2$.

$(iv)$ Let $\overline{G}$ be the complement of a graph $G\in
\mathcal {F}_{3s}(n)$, and let $\overline{G}\in \mathcal
{F}_{3s}(n)$. Then
$\chi\,''_{\varepsilon}(G)+\chi\,''_{\lambda}(\overline{G})\geq
\chi\,''_{\varepsilon}(K_n)$ for $\varepsilon=(s),(as),2s,2as$.
\end{lem}

It is noticeable, a complete graph $K_n$ admits no v-partially
$k$-vdtcs (also, v-partially $k$-avdtcs) at all.

\begin{lem} \label{thm:xxxxxxxx}
A graph $G\in \mathcal {F}_{3s}(n)$ admits a v-partially $k$-vdtc
(resp. a v-partially $k$-avdtc) if and only if $N(u)\cup \{u\}\neq
N(v)\cup \{v\}$ for distinct $u,v\in V(G)$ (resp. every edge $uv\in
E(G)$).
\end{lem}
\begin{proof} To show the proof of `if', we take a v-partially
$k$-vdtc $f$ of $G$. Then, $C\langle f,u\rangle\neq C\langle
f,v\rangle$ for distinct $u,v\in V(G)$. If $uv\in E(G)$, $C\langle
f,u\rangle\neq C\langle f,v\rangle$ means $\{f(x): x\in
N(u)\}\setminus\{f(u),f(v)\}\neq \{f(x): x\in
N(v)\}\setminus\{f(u),f(v)\}$, and furthermore $N(u)\cup \{u\}\neq
N(v)\cup \{v\}$. If $uv\not\in E(G)$, $C\langle f,u\rangle\neq
C\langle f,v\rangle$ means $N(u)=N(v)$, or $N(u)\neq N(v)$. No mater
one of two cases occurs, we have $N(u)\cup \{u\}\neq N(v)\cup
\{v\}$.

Conversely, it is straightforward to provide a total coloring $h$ of
$G$ for the proof of `only if'. In fact, we can set $h$ as a
bijection from $V(G)\cup E(G)$ to $[1,n+q]$, where $q=|E(G)|$.
Clearly, $C\langle h,u \rangle \neq C\langle h,v \rangle$ for
distinct $u,v\in V(G)$ (including every edge $uv\in E(G)$), since
$N(u)\cup \{u\}\neq N(v)\cup \{v\}$.
\end{proof}

\begin{lem} \label{thm:observation111}
$(i)$ Every graph $G\in \mathcal {F}_{3s}(n)$ admits $(6)$-total
colorings and holds  $\chi\,''_{(6)}(G)\geq \chi\,''_{\lambda}(G)$
for $\lambda=(s),(as)$, $s$, $as$, $2s$, $2as$.

$(ii)$ If a graph $G\in \mathcal {F}_{3s}(n)$ admits
$(8)$-distinguishing total colorings, then $\chi\,''_{(8)}(G)\geq
\chi\,''_{\varepsilon}(G)$ for $\varepsilon=(s),(as)$, $\langle s\rangle$, $
\langle as\rangle$, $s$, $as$, $2s$, $2as$. Furthermore,
$$\chi\,''_{(8)}(G)\leq \min\{\chi\,'_s(G)+\chi\,''_{\langle s\rangle}(G),~  \chi\,''_{(s)}(G)+\chi\,''_{\langle
s\rangle}(G)\}.$$
\end{lem}
\begin{proof} Notice that the assertion $(ii)$ is an immediate consequence of Definition \ref{def:N2-mu-e-mixed-distinguishing-coloring}.
Thereby ,we show the assertion $(i)$. Let $V(G)=\{u_i:i\in [1,n]\}$,
$V(G)=\{e_j:j\in [1,q]\}$, where $q=|E(G)|$. It is sufficient to
define a total coloring $f$ of $G$ as: $f(u_i)=i$ for $i\in [1,n]$,
and $f(e_j)=n+j$ for $j\in [1,q]$. Clearly, $f$ is a $(6)$-total
coloring. Definition \ref{def:a-new-coloring-contain-9-colorings}
leads to $\chi\,''_{(6)}(G)\geq \chi\,''_{\varepsilon}(G)$ for
$\varepsilon=(s),(as)$, $s$, $as$, $2s$, $2as$.
\end{proof}

\subsection{Connections between known chromatic numbers}

We will build up some connections between known chromatic numbers ($\chi$, $\chi'$, $\chi''$, $\chi'_s$, $\chi\,'_{as}$, $\chi\,''_{as}$) and the new chromatic numbers ($\chi \,''_{\langle s\rangle}$, $\chi \,''_{\langle as\rangle}$, $\chi\,''_{2s}$, $\chi\,''_{2as}$, $ \chi \,''_{(as)}$, $\chi \,''_{(s)}$) defined in Definition  \ref{def:22} and Definition  \ref{def:33} in this subsection.

\begin{lem} \label{thm:N2-vdtc-basic-lemma}
Let $G\in \mathcal {F}_{3s}(n)$.

$(i)$ $\chi\,''_{2s}(G)\leq \chi \,''_{(s)}(G)+\chi(G)$ and
$\chi\,''_{2s}(G)\leq \chi \,''_{\langle s\rangle}(G)+\chi\,'(G)$;
$\chi\,''_{2as}(G)\leq \chi \,''_{(as)}(G)+\chi(G)$ and
$\chi\,''_{2as}(G)\leq \chi \,''_{\langle as\rangle}(G)+\chi\,'(G)$.

$(ii)$ $\chi\,''_{\varepsilon}(G)\leq \chi \,'_{s}(G)+\chi(G)$ for
$\varepsilon=(s),(as),s,as,2s,2as$ and $\chi\,''_{\mu}(G)\leq \chi
\,'_{as}(G)+\chi(G)$ for $\mu=(as),as,2as$.
\end{lem}
\begin{proof} (1) Notice that the method for showing $\chi\,''_{2s}(G)\leq \chi
\,''_{(s)}(G)+\chi\,'(G)$ can be used to show the rest three
inequalities in the assertion $(i)$. Let $f$ be an e-partially
$k$-vdtc of $G$ with $k= \chi \,''_{(s)}(G)$. Under this coloring we
have $C(f,u)\neq C(f,v)$ for distinct $u,v\in V(G)$. Now we define
another total coloring $g$ of $G$ as: $g(e)=f(e)$ for $e\in E(G)$,
$g(u)=k+f\,'(u)$ for $u\in V(G)$, where $f\,'$ is a proper vertex
$\chi\,'(G)$-coloring of $G$. Notice that $\{g(x):x\in V(G)\}\cap
\{g(e):e\in E(G)\}=\emptyset$. Hence, $C(f,u)\neq C(f,v)$ means that
$N_2[g,u]\neq N_2[g,v]$ for distinct $u,v\in V(G)$.

(2) Since the proofs of two inequalities in the assertion $(ii)$
are very similar, we show only the first one. Write $\chi
\,'_{s}=\chi \,'_{s}(G)$ and $\chi=\chi(G)$. We color the edges of
$G$ with colors of $C=[1,\chi \,'_{s}]$ such that two incident edges
$e$ and $e\,'$ of $G$ are assigned distinct colors, and for any two
vertices $u$ and $v$ of $G$ the set of the colors assigned to the
edges being incident to $u$ is not equal to the set of the colors
assigned to the edges being incident to $v$. We, now, define  a
coloring $f$ for all vertices of $G$ with colors of $C\,'=\{\chi
\,'_{s}+1,\chi \,'_{s}+2,\dots ,\chi \,'_{s}+\chi\}$ such that the
color assigned to $u$ is different to the color assigned to $v$ for
arbitrary two vertices $u$ and $v$ of $V(G)$. From $C\cap
C\,'=\emptyset$ and the coloring way used above, we obtain
$C(f,u)\neq C(f,v)$, $C[f,u]\neq C2[f,v]$ and $N_2[f,u]\neq
N_2[f,v]$ for distinct $u,v\in V(G)$. Hence, $f$ is a mixed
$(6)$-total coloring, and moreover $\chi\,''_{\varepsilon}(G)\leq
\max\{f(x): x\in E(G)\cup V(G)\}=\chi \,'_{s}+\chi$ for
$\varepsilon=(s),(as),s,as,2s,2as$.
\end{proof}

Let $K_{m,n}$ be a complete bipartite graph for $m\geq n\geq 2$. By
Lemma \ref{thm:observation000} we have $\chi\,''_{2s}(K_{m,n})\leq
m+4$ if $m=n$ and $\chi\,''_{2s}(K_{m,n})\leq m+3$ if $m>n\geq 2$,
since $\chi(K_{m,n})=2$, and  $\chi \,'_{s}(K_{m,n})=m+2$ for $m=n$
and $\chi \,'_{s}(K_{m,n})=m+1$ for $m>n\geq 2$ (cf.
\cite{Burris-Schelp}). For example, $\chi\,''_{2s}(K_{3,2})=\chi
\,'_{s}(K_{3,2})+\chi(K_{3,2})=4+2$. Notice that we have used the
colors on the vertices being divided completely from the colors of
the edges of $G$ in the proof of Lemma
\ref{thm:N2-vdtc-basic-lemma}, so the upper bounds of
$\chi\,''_{2s}(G)$ in Lemma \ref{thm:N2-vdtc-basic-lemma} are not
optimal. A subset $V^{*}$ of $V(G)$ is an edge-covering set if any
edge $uv$ of $G$ holds $u\in V^{*}$ or $v\in V^{*}$. $G[V^{*}]$ indicates a
vertex induced subgraph of $V(G)$ such that two ends of each
edge of $G[V^{*}]$ both are in $V^{*}$.

\begin{lem} \label{thm:c4-vdtc-theorem-4}
Let $V^{*}$ be a smallest edge-covering set of a graph $G\in
\mathcal {F}_{3s}(n)$. Then there exists a bipartite subgraph $H$ with
bipartition $(V^{*},V(G)\setminus V^{*})$. We have
$\chi\,''_{2s}(G)\leq \chi\,''(G[V^{*}])+\chi \,'_{s}(H)+1$ and
$\chi\,''_{2as}(G)\leq \chi\,''(G[V^{*}])+\chi \,'_{as}(H)+1$ if
$H\in \mathcal {F}_{3s}(n)$.
\end{lem}
\begin{proof} We take a smallest edge-covering set $V^{*}$
of a graph $G\in \mathcal {F}_{3s}(n)$. Then  any $u\in
V\,'=V(G)\setminus V^{*}$ is adjacent to a vertex $v$ of $V^{*}$,
and $V\,'$ is an independent set of $G$ by the definition of an
edge-covering set.

Clearly, there exists a bipartite subgraph $H$ of $G$ that is defined as
$V(H)=V(G)$ and $E(H)=\{uv: u\in V^{*},\ v\in V\,'\}$. Write
$m=\chi\,''(G[V^{*}])$ and $t=\chi \,'_{s}(H)$. Suppose that $H\in
\mathcal {F}_{3s}(n)$. We use the colors of $S_1=[1,m]$ to color
properly every elements of the vertex-induced graph $G[V^{*}]$ and
color every vertex of $V\,'$ with color $m+1$. Consequently, we use
the colors of $S_2=\{m+2,m+3,\dots , m+1+t\}$ to color properly all
edges of $H$ such that for any two vertices $u$ and $v$ of $H$, the
set of the colors assigned to the edges incident to $u$ differs to
the set of the colors assigned to the edges incident to $v$. Hence,
we obtain a proper total coloring $f$ of $G$ such that $N_2[f,u]\neq
N_2[f,v]$ for distinct $u,v\in V(G)$ since $S_1\cap S_2=\emptyset$
and $H$ is a spanning subgraph of $G$, and confirm
$\chi\,''_{2s}(G)\leq m+t+1$.

The above method is valid for proving the second inequality of the
lemma.
\end{proof}

\subsection{New chromatic numbers and maximum degrees of graphs}

\begin{thm} \label{thm:xxxxxxxxxx}
For a bipartite graph $G\in \mathcal {F}_{3s}(n)$, we have
$\chi\,''_{(as)}(G)\leq \Delta (G)+3$; and $\chi\,''_{(s)}(G)\leq
\chi\,'_{s}(G)+\varepsilon$ for $\varepsilon=1$ if
$\chi\,'_{s}(G)>\Delta(G)$, $\varepsilon=2$ otherwise.
\end{thm}
\begin{proof} In the article \cite{Balister2}, a bipartite graph $G$ holds $\chi
\,'_{as}(G)\leq \Delta(G)+2$. Let $(X,Y)$ be the bipartition of the
bipartite graph $G$ and let $f$ be an e-partially $k$-avdtc of $G$
with $k=\Delta(G)+2$. We define another total coloring $g$ of $G$
as: $g(xy)=f(xy)$ for $xy\in E(G)$, $g(y)=k+1$ for $y\in Y$, and
$g(x)\in [1,k]\setminus \{f(xu):u\in N(x)\}$ for $x\in X$. Clearly,
$g$ is an e-partially $(k+1)$-avdtc of $G$, which means
$\chi\,''_{(as)}(G)\leq \Delta (G)+3$. By the same way used above we
can show the second inequality.
\end{proof}

\begin{thm} \label{thm:basic-results-mu-colorings}
For a graph $G\in \mathcal {F}_{3s}(n)$ with the independent number
$\alpha(G)$, we have

$(i)$ $\chi\,''_{2s}(G)\leq 4\Delta (G)$ if  $G\not \in \{C_{2m+1},K_n\}$.

$(ii)$ $\chi\,''_{2s}(G)\leq 2\Delta (G)+5$ if
$\delta(G)>\frac{n}{3}$ and $G\not \in \{C_{2m+1},K_n\}$.

$(iii)$ $\chi\,''_{2s}(G)\leq n-\alpha(G)+\chi \,'_{s}(H)+2$, where
$H$ is a bipartite subgraph with bipartition $(V^*,V(G)\setminus
V^*)$ generated from a smallest edge-covering set $V^*$.

$(iv)$ If diameter $D(G)\geq 3$, then $\chi\,''_{2s}(G)\geq m$ where
$m$ is a the smallest integer such that ${{m}\choose \delta+1}\geq
\ud_G(u)+\ud_G(v)+2$ for each pair of vertices $u,v$ holding
$\text{d}(u,v)\geq 3$.

$(v)$ $\chi\,''_{2s}(G)\leq n-\alpha(G)+\Delta(G)+7$ if
$\delta(G)>\frac{2}{3}n$ and $G\not \in \{C_{2m+1},K_n\}$.

$(vi)$ $\chi\,''_{2s}(G)\leq 2n-1$ if $G=K_n$.

$(vii)$ $\chi\,''_{2s}(G)\leq \chi\,''_{2s}(P_n)+\Delta(G)$ if $G$
is hamiltonian.

$(viii)$ $\chi\,''_{2s}(G)\leq n_1(T)+\Delta(G)+1$ if $G$ contains a
spanning tree $T$ with $n_2(T)=0$.
\end{thm}
\begin{proof} $(1)$ In the article \cite{Akbafi-Bidkhori-Nosrati}, the authors show that a simple
graph $G$ without isolated edges holds $\chi \,'_{s}(G)\leq 3\Delta
(G)$. By the Brooks' theorem, we obtain the assertion $(i)$.

$(2)$ Since $G\not \in \{C_{2m+1},K_n\}$, the assertion $(ii)$
follows from $\chi(G)\leq \Delta(G)$ and a result of the article
\cite{Bazgan-Benhamdine-Li-Wozniak}: if $\delta(G)>\frac{1}{2}|G|$,
then $\chi \,'_{s}(G)\leq \Delta (G)+5$.

$(3)$ Let $\alpha(G)$ and $\beta(G)$ be the vertex-independent
number and the edge-covering number of $G$, respectively. Therefore,
$\alpha(G)+\beta(G)=|G|$ (cf. \cite{Bondy-Murty-new}). Notice that
$|G[V^{*}]| =\beta(G)$ and $\chi\,''(G[V^{*}])\leq \beta(G)+1$,
since $\chi\,''(K_{2m})=\chi\,''(K_{2m+1})=2m+1$. By Lemma
\ref{thm:c4-vdtc-theorem-4} we obtain the assertion $(iii)$, that is,
$\chi\,''_{2s}(G)\leq \beta(G)+\chi \,'_{s}(H)+2$, or
$\chi\,''_{2s}(G)\leq |V(G)|-\alpha(G)+\chi \,'_{s}(H)+2$.

$(4)$ Notice that $\text{d}(u,v)\geq 3$ for distinct $u,v\in V(G)$.
It is clear that $|N(u)\cup N(v)|=\ud_G(u)+\ud_G(v)$, this shows
that we need at least $\ud_G(u)+\ud_G(u)+2$ distinct colors. The
smallest case is ${{m}\choose \delta+1}\geq \ud_G(u)+\ud_G(v)+2$, so
the assertion $(iv)$ holds true, as desired.

$(5)$ To show that $G$ contains a bipartite spanning graph $H$ with
$\ud_H(u)\geq \frac{1}{2}\ud_G(u)$ for every $u\in V(H)=V(G)$, we
make a partition $V(G)=S_1\cup S_2$ for $S_1\cap S_2=\emptyset $
such that the cardinality of subset $E\,'=\{uv: u\in S_1, v\in
S_2\}$ is as large as possible, and obtain an edge induced graph
$G[E\,']$ over $E\,'$. If it is not that $2\ud _H(u)\geq \ud _G(u)$
for every vertex $u\in V(H)=V(G)$. Hence, there is a vertex $v_0\in
S_1$ such that $2\ud _H(v_0)< \ud _G(v_0)$, so we can take
$S\,^*_1=S_1\setminus \{v_0\}$ and $S\,^*_2=S_2\cup \{v_0\}$, and
get $E\,^*=\{uv: u\in S\,^*_1, v\in S\,^*_2\}$ such that
$|E\,^*|>|E\,'|$; a contradiction with the choice of $E\,'$. This
bipartite graph $H$ was discovered first by Erd\"{o}s. Since
$\ud_H(u)\geq \frac{1}{2}\ud_G(u)\geq \frac{n}{3}$ and $\chi(G)\leq
\Delta(G)$ according to $G\not \in \{C_{2m+1},K_n\}$, the assertion
$(v)$ follows from the assertion $(iii)$, Lemma
\ref{thm:c4-vdtc-theorem-4} and a result of the article
\cite{Bazgan-Benhamdine-Li-Wozniak} stated in $(P2)$ above.

$(6)$ Let $T$ be a spanning tree of $G$. Then
\begin{equation}\label{eqa:new-results-spanningtree-subgraphs}
\chi\,''_{2s}(G)\leq \chi\,''_{2s}(T)+\chi\,'(G-E(T)),
\end{equation}
and
\begin{equation}\label{eqa:new-results-spanningtree-subgraphs11}
\Delta(G-E(T))\leq \Delta(G)-1,\quad \chi\,'(G-E(T))\leq \Delta(G)
\end{equation}
by Vizing's theorem. If $G=K_n$, we have a star $T=K_{1,n-1}$.
Furthermore, $\chi\,''_{2s}(K_{1,n-1})=n$, $\chi\,'(G-E(T))=n-1$.
The assertion $(vi)$ follows from
(\ref{eqa:new-results-spanningtree-subgraphs}).

$(7)$ Notice that $G$ contains a Hamilton path $P_n$ which is a
spanning tree of $G$, and $\chi\,'(G-E(P_n))\leq \Delta(G)$ by
(\ref{eqa:new-results-spanningtree-subgraphs11}). The assertion
$(vii)$ follows from (\ref{eqa:new-results-spanningtree-subgraphs}).

$(8)$  If $G$ contains a spanning tree $T$ with $n_2(T)=0$, then
$n_1(T)\leq \chi\,''_{2s}(T)\leq n_1(T)+1$.  The last assertion
follows from both
(\ref{eqa:new-results-spanningtree-subgraphs}) and
(\ref{eqa:new-results-spanningtree-subgraphs11}).

The proof of the theorem is complete.
\end{proof}

\begin{thm} \label{thm:some-results-mu-e-colorings}
For a graph $G\in \mathcal {F}_{3s}(n)$ we have

$(i)$ $\chi\,''_{2as}(G)\leq 8$ if $\Delta(G)\leq 3$ and $G\not \in
\{C_{2m+1},K_3\}$.

$(ii)$ $\chi\,''_{2as}(G)\leq \Delta(G)+4$ if $G$ is bipartite.

$(iii)$ $\chi\,''_{2as}(G)\leq 2\Delta(G)+3$ if $G$ is a
$3$-colorable, Hamilton graph, and $G\not \in \{C_{2m+1},K_n\}$.

$(iv)$ $\chi\,''_{2as}(G)\leq 2\Delta(G)+2$ if $G$ is a planar graph
$G$ with girth $g\geq 6$ and $\Delta (G) \geq 3$.

$(v)$ $G$ has a spanning subgraph $G^*$ with $|E(G^*)|<
\frac{1}{2}|E(G)|$ such that $$\chi\,''_{2as}(G)\leq
\chi\,''(G^*)+\Delta(G)-\delta(G^*)+2.$$
\end{thm}
\begin{proof}  In the article \cite{Balister2}, the authors shown that
$\chi \,'_{as}(G)\leq 5$ if $\Delta(G)\leq 3$, and $\chi
\,'_{as}(G)\leq \Delta(G)+2$ if $G$ is bipartite. By Lemma
\ref{thm:N2-vdtc-basic-lemma} we obtain the assertions $(i)$ and
$(ii)$.

The assertion $(iii)$ follows from Lemma
\ref{thm:N2-vdtc-basic-lemma} and a result of the article
\cite{Bin-Liu-Guizhen-Liu}: every connected $3$-colourable
Hamiltonian graph $G$ holds $\chi\,'_{as}(G)\leq \Delta(G) + 3$.

Wang and Wang \cite{Weifan-Wang-Yiqiao} distribute that a
planar graph $G$ with girth $g\geq 6$ and $\Delta (G) \geq 3$ holds
$\chi\,'_{as}(G)\leq \Delta(G)+2$. This result and Lemma
\ref{thm:N2-vdtc-basic-lemma} induce the assertion $(iv)$.

To show that assertion $(v)$, we apply that a certain bipartite spanning subgraph $H$ of $G$ with
$\ud_H(u)\geq \frac{1}{2}\ud_G(u)$ for every $u\in V(H)$ exists according to
the proof of the assertion $(v)$ of Theorem
\ref{thm:basic-results-mu-colorings}. So, $|E(H)|\geq
\frac{1}{2}|E(G)|$, which means $|E(G^*)|< \frac{1}{2}|E(G)|$, where
$G^*=G-E(H)$. Let $f$ be a proper total coloring of the graph
$G^*=G-E(H)$ with the color set $[1,\chi\,''(G^*)]$, and let $h$ be
a adjacent vertex distinguishing edge coloring of $H$ with the color
set $\{\chi\,''(G^*)+1,\chi\,''(G^*)+2,\dots
,\chi\,''(G^*)+\chi\,'_{as}(H)\}$. Both colorings give
$\chi\,''_{2as}(G)\leq \chi\,''(G^*)+\chi\,'_{as}(H)\leq
\chi\,''(G^*)+\Delta(H)+2$ since $H$ is bipartite and $\chi
\,'_{as}(H)\leq \Delta(H)+2$ (cf. \cite{Balister2}). Notice that
$\Delta(H)=\Delta(G)-\delta(G^*)$. We obtain the assertion $(v)$.

This theorem is covered.
\end{proof}

\subsection{Construction of graphs having total colorings with
additional constrained conditions}

\begin{thm} \label{thm:new-results-mothergraph-subgraphs}
Suppose that a graph $H\in \mathcal {F}_{3s}(n)$ is not a complete graph,
then the graph $G$ obtained by adding an edge of the complement of
$H$ to $H$ holds $\chi\,''_{2s}(G)\leq \chi\,''_{2s}(H)+1$.
\end{thm}
\begin{proof} Let $f$ be a $\mu(k)$-coloring of a non-complete graph $H\in
\mathcal {F}_{3s}(n)$ with $k=\chi\,''_{2s}(H)$. Suppose that $u$ is
not adjacent to $v$ in $H$. We have a graph $G=H+uv$, and define a
coloring $g$ of $G$ in the following cases.

Case 1. $f(u)\neq f(v)$. We set $g(uv)=k+1$, $g(z)=f(z)$ for $z\in
V(G)\cup (E(G)\setminus \{uv\})$. Notice that $N_2[g,u]\setminus
\{k+1\}=N_2[f,u]\neq N_2[f,v]=N_2[g,v]\setminus \{k+1\}$. We can
confirm $N_2[f,x]\neq N_2[f,y]$ for distinct $x,y\in V(G)$. Thereby,
$g$ is a $\mu(k)$-coloring of $G$, and furthermore
$\chi\,''_{2s}(G)\leq \chi\,''_{2s}(H)+1$.

Case 2. $f(u)=f(v)$, $N_2[f,u]\not\subset  N_2[f,v]$ and
$N_2[f,v]\not\subset  N_2[f,u]$. Notice that there is a color
$\alpha\in N_2[f,u]$, but $\alpha\not\in N_2[f,v]$; and there is a
color $\beta\in N_2[f,v]$, but $\beta\not\in N_2[f,u]$.

Case 2.1. $f(uu_i)\neq \alpha$ for $u_i\in N(u)$. We set
$g(uv)=\alpha$, $g(v)=k+1$, and $g(z)=f(z)$ for $z\in
(V(G)\setminus\{v\})\cup (E(G)\setminus \{uv\})$. Notice that
$k+1\in N_2[g,u]$ and $k+1\in N_2[g,v]$. Since $\beta\not\in
N_2[f,u]$, so $\beta\not\in N_2[g,u]$, it follows that $g$ is a
$\mu(k)$-coloring of $G$.

Case 2.2. $f(uu_i)=a$ for some $u_i\in N(u)$. Define $f\,'$ of
$H$ as:  $f\,'(z)=k+1$ if $f(z)=\alpha$ and $f\,'(z)=f(z)$
if$f(z)\neq a$, $z\in V(H)\cup E(H)$.  Notice that $f\,'$ is a
$\mu(k)$-coloring of $H$. Next, we set $g(uv)=a$, $g(v)=k+1$,
and $g(z)=f\,'(z)$ for $z\in (V(G)\setminus\{v\})\cup (E(G)\setminus
\{uv\})$. We can see  $a\not \in N_2[g,z]$ for $z\in
V(G)\setminus \{u,v\}$, and $\beta\not\in N_2[f,u]$. This coloring
$g$ gives $\chi\,''_{2s}(G)\leq \chi\,''_{2s}(H)+1$.

Case 3. $f(u)=f(v)$ and $N_2[f,u]\subset  N_2[f,v]$. Notice that
$N_2[f,v]\setminus N_2[f,u]\neq \emptyset$, so we can obtain the
desired $\mu(k)$-coloring $g$ of $G$ by the methods in Case 2.1 and
Case 2.2.

We complete the proof of the theorem.
\end{proof}

\begin{thm} \label{thm:add-vertex-and-edges}
$(i)$ Adding a leaf to a connected graph $H$ produces a graph $G$
holding $\chi\,''_{2s}(G)\leq \chi\,''_{2s}(H)+1$.

$(ii)$ If a graph $G$ obtained by deleting a vertex of degree $m$
from a connected graph $H$ is connected. Then $\chi\,''_{2s}(G)\leq
\chi\,''_{2s}(H)+m$.

$(iii)$ Adding a leaf $v_i\not\in V(H)$ to a vertex $u_i$ of a
connected graph $H$ for $i\in [1,m]$ produces a connected graph $G$
holding $\chi\,''_{2as}(G)\leq \chi\,''_{2as}(H)+1$.
\end{thm}
\begin{proof} $(1)$ Let $f$ be a $\mu(k)$-coloring of a connected graph
$H$ with $k=\chi\,''_{2s}(H)$. We add a leaf $v$ to $H$ by joining
$v$ and a vertex $u\in V(G)$, and the resulting graph is denoted as
$G$. We define a coloring $g$ of $G$ as: $g(w)=f(w)$ for $w\in
(V(G)\cup E(G))\setminus\{v,uv\}$; $g(uv)=k+1$, and $g(v)=f(u\,')$,
where $u\,'\in N(u)$. Notice that $N_2[g,v]=\{f(u\,'),f(u),k+1\}$ is
a proper subset of $N_2[g,u]$ since $f(uu\,')\not\in N_2[g,v]$. We
see that $g$ is a $\mu(k+1)$-coloring of $G$, which induces the
assertion $(i)$.

$(2)$ Let $H$ be a connected graph, and let $G=H-w$ be connected
with $\ud_H(w)=m$ and $N(w)=\{w_1,w_2,\dots,w_m\}$. By the assertion
$(i)$ a connected graph $G_1$ by adding a leaf $w$ to $G$ through
joining $w$ and a vertex $w_1\in V(G)$ holds $\chi\,''_{2s}(G_1)\leq
\chi\,''_{2s}(H)+1$. Applying Theorem
\ref{thm:new-results-mothergraph-subgraphs} repeatedly $(m-1)$ times
by joining $w$ and each $w_i\in N(w)\setminus \{w_1\}$, we get
$\chi\,''_{2s}(G)\leq \chi\,''_{2s}(H)+m$, as desired.

$(3)$ Let $h$ be a $\mu_e(k)$-coloring of a connected graph
$H$ with $k=\chi\,''_{2as}(H)$. Take distinct vertices $v_i\not\in
V(H)$, and select arbitrarily distinct vertices $u_i\in V(H)$, $i\in
[1,m]$. We have a graph $G$ obtained by joining $v_i$ and $u_i$ by
an edge for $i\in [1,m]$, and define a coloring $\beta$ of $G$ in
the way that $\beta(z)=h(z)$ for $z\in (V(G)\cup
E(G))\setminus\{v_i,u_iv_i:i\in [1,m]\}$; $\beta(u_iv_i)=k+1$ and
$\beta(v_i)=h(u_iu\,'_i)$ for $i\in [1,m]$, where $u\,'_i\in
N(u_i)$. It follows that the coloring $\beta$ is a
$\mu_e(k)$-coloring of $G$, since $N_2[\beta,x]\neq N_2[\beta,y]$ for every
edge $xy\in E(G)$.
\end{proof}

\section{Problems for further works}

As further works, we propose the following problems:

\begin{prob} \label{prob:conjecture}
(1) If $G_1\subseteq H\subseteq G_2$ and
$\chi\,''_{\lambda}(G_1)=\chi\,''_{\lambda}(G_2)=k$, do we have
$\chi\,''_{\lambda}(H)=k$ for $\lambda=2s,2as$?

(2) Let $D(G)=2$. If the smallest number $k$ satisfies that
${{k}\choose \Delta+1}\geq n$, then $|\chi\,''_{2s}(G)-k|\leq 1$?

(3) Suppose that $D(G)\geq 3$ and $\delta(G)\neq \Delta(G)$. If
${{k}\choose \Delta+1}\geq n$ and ${{m}\choose \delta+1}\geq n$ and
$m\leq k $, then $m\leq \chi\,''_{2s}(G)\leq k$?

(4) Characterize simple graphs $G$ such that $N(u)\cup \{u\}\neq
N(v)\cup \{v\}$ for distinct $u,v\in V(G)$ (resp. for every edge
$uv\in E(G)$).
\end{prob}

\begin{conj} \label{prob:conjecture}
Every connected, simple graph $G$ holds $\chi\,'_{s}(G)\leq
\chi\,''_{(s)}(G)\leq \chi\,'_{s}(G)+1$ and $\chi\,'_{as}(G)\leq
\chi\,''_{(as)}(G)\leq \chi\,'_{as}(G)+1$.
\end{conj}

It has been discovered that there are infinite simple graphs $G$ and
some their proper subgraphs $H$ such that $\chi\,'_{\lambda}(H)>\chi\,'_{\lambda}(G)$ and
$\chi\,''_{\lambda}(H)>\chi\,''_{\lambda}(G)$ for $\lambda=s,as$ (cf. \cite{Yao-Chen-Yao-Wang2012}). But, we have
\begin{conj} \label{prob:conjecture}
No simple graph $G$ and its proper subgraphs $H$ hold
$\chi\,''_{(8)}(G)<\chi\,''_{(8)}(H)$ true.
\end{conj}

\vskip 0.6cm

\noindent \textbf{Acknowledgment.} The author, \emph{Bing Yao}, thanks the National Natural Science
Foundation of China under grants No. 61163054 and No. 61163037. The second author, \emph{Ming Yao}, thanks The Special Funds of Finance Department of Gansu Province of China under grant No. 2014-63. The third author, \emph{Xiang-en Chen}, thanks the National Natural Science Foundation of China under grant No. 61363060.

\end{document}